\documentclass[12pt]{amsart}

\usepackage{amsmath}
\usepackage{amssymb}
\usepackage{amsthm}
\usepackage[OT2,T1]{fontenc}

\newtheorem{thm}{Theorem}[subsection]
\newtheorem{lem}[thm]{Lemma}
\newtheorem{prop}[thm]{Proposition}
\newtheorem{cor}[thm]{Corollary}
\newtheorem*{thmX}{Theorem X}
\newtheorem*{thmY}{Theorem Y}
\newtheorem*{thmZ}{Theorem Z}

\theoremstyle{definition}

\newtheorem{example}[thm]{Example}
\newtheorem*{question}{Questions}

\theoremstyle{remark}
\newtheorem{rem}[thm]{Remark}
\newtheorem*{acknowledgments}{Acknowledgments}

\newcommand\Gal{\mathrm{Gal}}
\newcommand\rank{\mathop{\mathrm{rank}_{\mathbb{Z}}}}

\newcommand\Hom{\mathop{\mathrm{Hom}}}
\newcommand\Char{\mathop{\mathrm{Char}_{\Lambda}}}
\newcommand\Sha{\text{\usefont{OT2}{wncyr}{m}{n} Sh}}

\begin{document}

\title{On the unramified Iwasawa module of a $\mathbb{Z}_p$-extension generated by 
division points of a CM elliptic curve}
\author{Tsuyoshi Itoh}
\date{\today}

\maketitle

\begin{abstract}
We consider the unramified Iwasawa module $X (F_\infty)$ of a 
certain $\mathbb{Z}_p$-extension $F_\infty/F_0$ 
generated by division points of an elliptic 
curve with complex multiplication.
This $\mathbb{Z}_p$-extension has properties similar to those of the cyclotomic $\mathbb{Z}_p$-extension 
of a real abelian field, however, 
it is already known that $X (F_\infty)$ can be infinite.
That is, an analog of Greenberg's conjecture for this $\mathbb{Z}_p$-extension fails.
In this paper, we mainly consider analogs of weak forms of Greenberg's conjecture.
\end{abstract}

\section{Introduction}\label{sec_intro}

\subsection{Our questions}\label{sub_questions}
First at all, we explain the situation which we will treat.
In this paper (except for Appendix \ref{sec_appendix}), 
we shall consider the following situation: 
\begin{itemize}
\item[(C1)] $K$ is an imaginary quadratic field whose class number is $1$, 
\item[(C2)] $p$ is an odd prime number which splits two distinct primes 
$\mathfrak{p}$ and $\overline{\mathfrak{p}}$ in $K$, 
\item[(C3)] $E$ is an elliptic curve over $\mathbb{Q}$ which has complex multiplication 
by the ring $O_K$ of integers of $K$, and $E$ has good reduction at $p$.
\end{itemize}
In the following, we assume that $K$, $p$, $E$ satisfy (C1), (C2), (C3).
Many authors treated this situation (or similar situations).
See, e.g., \cite{C-W}, \cite{BGS}, \cite{Per}, \cite{Gil87}, \cite{Sch}, 
\cite{Rub91m}, \cite{Rub91o}, \cite[Section 5]{Gre01}, \cite{F-K}, etc.

We shall recall several known facts 
(see, e.g., \cite{C-W}, \cite{BGS}, \cite[pp.364--365]{Gre01}, \cite[Section 1]{F-K}).
Let $\psi$ be the Gr{\"o}ssencharacter of $E$ over $K$, and put 
$\pi = \psi (\mathfrak{p})$.
Then, $\pi$ is a generator of the principal ideal $\mathfrak{p}$.
For every non-negative integer $n$, 
let $E [\pi^{n+1}] \subset E (\overline{\mathbb{Q}})$ be the group of $\pi^{n+1}$-division points of $E$.
We put $F_n = K (E[\pi^{n+1}])$ for every $n$.
Then $F_n /K$ is an abelian extension, and $\mathfrak{p}$ is totally ramified in $F_n/K$.
We also put $F_\infty = \bigcup_{n} F_n$.
It is known that 
\[ \Gal (F_\infty / K) \cong \Delta \times \Gamma, \]
where $\Delta (\cong \Gal (F_0/K))$ is a cyclic group of order $p-1$ and 
$\Gamma (=\Gal (F_\infty /F_0))$ is 
topologically isomorphic to the additive group of $\mathbb{Z}_p$.
(We often identify $\Delta$ with $\Gal (F_0/K)$ via the natural restriction mapping.)
Let $\mathfrak{P}$ be the unique prime of $F_0$ lying above $\mathfrak{p}$.
Note that $F_\infty / F_0$ is a $\mathbb{Z}_p$-extension 
which is unramified outside $\mathfrak{P}$.

We denote by $L (F_\infty)/ F_\infty$ the maximal unramified abelian pro-$p$ extension 
and $M (F_\infty) /F_\infty$ the maximal abelian pro-$p$ extension 
unramified outside the unique prime lying above $\mathfrak{p}$.
We put $X (F_\infty) = \Gal (L (F_\infty)/ F_\infty)$ (the unramified Iwasawa module) and 
$\mathfrak{X} (F_\infty) = \Gal (M (F_\infty)/ F_\infty)$ 
(the $\mathfrak{p}$-ramified Iwasawa module).
We also put $\Lambda = \mathbb{Z}_p [[\Gamma ]]$.
Then, it is well known that $X (F_\infty)$ is a finitely generated torsion $\Lambda$-module.

We note that $\mathfrak{X} (F_\infty)$ is also a finitely 
generated torsion $\Lambda$-module because the 
``$\{ \mathfrak{P} \}$-adic analog'' of Leopoldt's conjecture  
holds for $F_0$ (see Section 2 for the detail).
Recall that a similar property holds for the ``$p$-ramified Iwasawa module'' of 
the cyclotomic $\mathbb{Z}_p$-extension of real abelian fields. 
(For these topics, see \cite{Gre78}).
We also mention that $\mathfrak{X} (F_\infty)$ is finitely generated 
as a $\mathbb{Z}_p$-module (see \cite{Gil87}, \cite{Sch}, \cite{O-V}), 
and a similar fact for the $p$-ramified Iwasawa module of 
the cyclotomic $\mathbb{Z}_p$-extension of real abelian fields
(with odd $p$) follows from Ferrero-Washington's theorem \cite{F-W} and Kummer duality.
Furthermore, the main conjecture holds for this situation 
(see \cite{Rub91m}, \cite{Rub91o}), 
and the statement is similar to that of the even part version of the main conjecture 
for abelian fields (rather than the odd part version).
(Cf., e.g., \cite{M-W}, \cite[Appendix by Karl Rubin]{Lang}.)
Hence, it might be said that $F_\infty /F_0$ is close to 
the cyclotomic $\mathbb{Z}_p$-extension of real abelian fields in some sense.
We would like to know how many properties these $\mathbb{Z}_p$-extensions have in common.

It is conjectured that the unramified Iwasawa module of the cyclotomic $\mathbb{Z}_p$-extension 
is finite for every totally real field (Greenberg's conjecture \cite{Gre76}).
On the other hand, it is known that $X (F_\infty)$ can be infinite in general.
We denote by $\rank E(\mathbb{Q})$ the free rank of the 
Mordell-Weil group $E (\mathbb{Q})$.

\begin{thmX}[see p.551, Remark of Rubin \cite{Rub87}, pp.364--366 of Greenberg \cite{Gre01}]
Assume that $K$, $p$, $E$ satisfy (C1), (C2), (C3). 
If $\rank E(\mathbb{Q}) \geq 2$, then $X (F_\infty)$ is not finite.
\end{thmX}

Hence, under Greenberg's conjecture, $F_\infty / F_0$ is different from 
the cyclotomic $\mathbb{Z}_p$-extension of a real abelian fields on this point.

We also note that ``weak forms'' of Greenberg's conjecture are considered by several authors
(see \cite{Kra}, \cite{Ich}, \cite{O-T98}, \cite{B-NQD}, \cite{NQD06}, \cite{NQD17I}, 
\cite{Fuj}, etc.).
Based on these studies, we shall consider the following questions.
These are analogs of weak forms of Greenberg's conjecture 
treated in \cite{NQD06}, \cite{NQD17I} (see also \cite{NQD17A}).
We denote by $X (F_\infty)_{\mathrm{fin}}$ the maximal finite $\Lambda$-submodule of 
$X (F_\infty)$.
\begin{itemize}
\item When $X (F_\infty)$ is not trivial, is
$X (F_\infty)_{\mathrm{fin}}$ not trivial?
\item When $X (F_\infty)$ is not trivial, is 
$\Gal (M (F_\infty) / L (F_\infty))$ not trivial?
\end{itemize}

\begin{rem}\label{rem_tame}
In \cite[Appendix A]{Itoh18}, 
similar questions for ``tamely ramified Iwasawa modules'' of 
the cyclotomic $\mathbb{Z}_p$-extension of a totally real field  
are considered.
See also \cite{F-I}.
\end{rem}

\begin{rem}\label{rem_Q1toQ2}
It is known that $\mathfrak{X} (F_\infty)$ does not have a non-trivial finite $\Lambda$-submodule
(see \cite[p.94]{Gre78}).
Hence if $X (F_\infty)_{\mathrm{fin}}$ is not trivial, then 
$\Gal (M (F_\infty) / L (F_\infty))$ is also not trivial 
(cf., e.g., \cite[Lemme 2.1]{NQD06}).
\end{rem}

Actually, it is already known that the second question has an 
affirmative answer for a large family of elliptic curves.

\begin{thmY}[see Lemma 35 of Coates-Wiles \cite{C-W}]
Assume that $K$, $p$, $E$ satisfy (C1), (C2), (C3).
If $\rank E(\mathbb{Q}) \geq 1$, 
then $\Gal (M (F_\infty) / L (F_\infty))$ is not trivial.
\end{thmY}

Strictly speaking, it was assumed that $p \geq 5$ in \cite{C-W}.
However, one can also show the same assertion for $p=3$ similarly.
This fact seems well known (see \cite[(11.6) Proposition]{Rub87}, \cite{Gre01}).
See also Section \ref{rem_ThB} 
for analogs of Theorems X and Y for the case when $p=2$. 

Let $\phi$ be the isomorphism $\Gal (F_\infty /K) \to \mathbb{Z}_p^\times$ 
which satisfy 
$P^{\sigma} = \phi(\sigma) P$ for all $P \in E[\pi^{n+1}]$ and $\sigma \in \Gal (F_\infty /K)$ 
(see, e.g., \cite[p.364]{Gre01}, \cite{F-K}).
Let $\chi$ be the restriction of $\phi$ on $\Delta$.
We denote by $X (F_\infty)^{\chi}$ (resp. $\mathfrak{X} (F_\infty)^{\chi}$)
the $\chi$-part of $X (F_\infty)$ (resp. $\mathfrak{X} (F_\infty)$).
(For a $\mathbb{Z}_p [\Delta]$-module $M$ appeared later, 
we also write $M^\chi$ for its $\chi$-part $M^\chi$.)
$X (F_\infty)^{\chi}$ and $\mathfrak{X} (F_\infty)^{\chi}$ are also 
considered as $\Lambda$-modules.
In this paper, we mainly treat the $\chi$-part version of the above questions.
\begin{question}
Assume that $K$, $p$, $E$ satisfy (C1), (C2), (C3).
\begin{itemize}
\item[(Q1)] When $X (F_\infty)^{\chi}$ is not trivial, is 
$X (F_\infty)_{\mathrm{fin}}^{\chi}$ not trivial?
\item[(Q2)] When $X (F_\infty)^{\chi}$ is not trivial, is 
$\Gal (M (F_\infty) / L (F_\infty))^{\chi}$ not trivial?
\end{itemize}
\end{question}

\begin{rem}\label{Gamma_coinvariant}
Let $A (F_0)$ be the Sylow $p$-subgroup of the ideal class group of $F_0$.
Since $\mathfrak{P}$ is the only prime which ramifies in $F_\infty/F_0$ and it is 
totally ramified, the $\Gamma$-coinvariant quotient 
$(X(F_\infty)^{\chi})_\Gamma$ is isomorphic to $A(F_0)^{\chi}$ 
(see, e.g., \cite[Theorem 5.1]{Rub91o}).
From this, we see that $X(F_\infty)^{\chi}$ is trivial if and only if 
$A(F_0)^{\chi}$ is trivial.
\end{rem}

Note that Theorems X and Y actually give the results for the $\chi$-part.
(See \cite[p.250]{C-W}, \cite[p.551, Remark]{Rub87}, \cite[p.365]{Gre01}.) 
\begin{thmZ}
Assume that $K$, $p$, $E$ satisfy (C1), (C2), (C3). 
\begin{itemize}
\item[(i)] If $\rank E(\mathbb{Q}) \geq 2$, 
then $X (F_\infty)^\chi$ is not finite.
\item[(ii)] If $\rank E(\mathbb{Q}) \geq 1$, 
then $\Gal (M (F_\infty) / L (F_\infty))^\chi$ is not trivial.
\end{itemize}
\end{thmZ}
Hence, (Q2) has an affirmative answer when $\rank E(\mathbb{Q}) \geq 1$.

\subsection{Organization of the present paper}
Our purposes of this paper are giving several criteria for (Q1), (Q2),  
and confirming these questions for specific elliptic curves.
In Section \ref{sec_criteria}, we will give the criteria.

We shall give examples in Section \ref{sec_examples}.
First, we remark that the examples given in Fukuda-Komatsu's paper \cite{F-K} 
are also examples for our questions (Section \ref{sec_F-K}).
We also give an example for (Q2) in Section \ref{sub_307}.
These examples are elliptic curves of the form $y^2 = x^3 - d x$ with $p=5$.
In Section \ref{sub_11}, we shall treat the elliptic curves of the form 
$y^2 = x^3 - 264 d^2 x + 1694 d^3$ with $p=3$.
Consequently, for the question (Q1), we found that the following cases actually exist.
\begin{itemize}
\item $X (F_\infty)^\chi$ is infinite and 
$X (F_\infty)_{\mathrm{fin}}^{\chi}$ is trivial
(i.e., (Q1) has a negative answer).  
\item $X (F_\infty)^\chi$ is infinite and 
$X (F_\infty)_{\mathrm{fin}}^{\chi}$ is non-trivial.
\item $X (F_\infty)^\chi$ is non-trivial and finite. 
\end{itemize}
On the other hand, for most of the cases which we examined, 
(Q2) has an affirmative answer (see, e.g., Section \ref{sub_11}).
In particular, there are examples such that (Q2) has an affirmative answer  
and (Q1) has an negative answer.

In Appendix \ref{sec_appendix}, we will treat the case when $p=2$. 
We shall consider similar questions (Q1t), (Q2t), and 
we will show that these questions are equivalent.

\subsection{Changes from the previous version}
(This subsection is written only in the arXiv version.)
The main change from \texttt{arXiv:2001.04687v6} (abbreviated as \texttt{v6}) 
is that the contents of Appendix A of \texttt{v6} was moved to Section 2.
Theorem A.1.1 of \texttt{v6} was moved to Remark \ref{rem_old}.
Note that the proof was slightly modified, and the result was extended to the case when $p \geq 3$.
Example A.3.1 was moved to Section \ref{sub_307}.
As a consequence, Appendix B of \texttt{v6} was renamed to Appendix A of this version.
Moreover various texts were modified mainly to shorten this paper.

\section{Criteria for the questions (Q1) and (Q2)}\label{sec_criteria}

\subsection{Preliminaries}\label{sub_criteria_pre}
Let the notation be as in Section \ref{sec_intro}.
We also define the following notation:
\begin{itemize}
\item $K_{\mathfrak{p}}$ : the completion of $K$ at $\mathfrak{p}$,
\item $(F_0)_{\mathfrak{P}}$ : the completion of $F_0$ at $\mathfrak{P}$,
\item $O_{\mathfrak{P}}$ : the valuation ring of $(F_0)_{\mathfrak{P}}$, 
\item $\mathcal{U}^i = 1 + \mathfrak{P}^i O_{\mathfrak{P}}$ (for $i=1,2$), 
\item $E (F_0)^1$ : the group of units of $F_0$ which are congruent to $1$ modulo $\mathfrak{P}$, 
\item $\mathcal{E}^1$ : the closure of $E (F_0)^1$ in $\mathcal{U}^1$,
\end{itemize}
By class field theory, we see that 
\[ \Gal (M (F_0)/L(F_0)) \cong \mathcal{U}^1/ \mathcal{E}^1. \]
Note that the $\{ \mathfrak{P} \}$-adic analog of Leopoldt's conjecture 
(for $F_0$) asserts 
that the $\mathbb{Z}_p$-rank of $\mathcal{E}^1$ is equal to the free rank of 
the group of global units of $F_0$ 
(for the name of this conjecture, we followed \cite{Gre78}).
Recall that this holds true since $F_0/K$ is an abelian extension 
(see \cite{Bru} and \cite{Gre78}).

We fix a topological generator $\gamma_0$ of $\Gamma$, and we shall identify 
$\Lambda$ with $\mathbb{Z}_p [[T]]$ ($\gamma_0 \mapsto 1+T$).
For a finitely generated torsion $\Lambda$-module $Y$, 
$Y^\Gamma$ denotes the $\Gamma$-invariant submodule of $Y$, 
$Y_\Gamma$ denotes the $\Gamma$-coinvariant quotient of $Y$, and 
$\Char Y$ denotes the characteristic ideal of $Y$.
For a finite group $B$, let $|B|$ be the order of $B$.

Since the $\{ \mathfrak{P} \}$-adic analog of Leopoldt's conjecture holds for $F_0$, 
we see that $\mathfrak{X} (F_\infty)^\chi_\Gamma$ is finite.
From this, we can deduce that the a generator of $\Char \mathfrak{X} (F_\infty)^\chi$ is 
not divisible by $T$ 
(the same result holds for $X (F_\infty)^\chi$). 
Moreover, $(\mathfrak{X} (F_\infty)^\chi)^\Gamma$ is trivial 
since $\mathfrak{X} (F_\infty)^\chi$ does not have a non-trivial finite 
$\Lambda$-module.
We can also show that 
$(X (F_\infty)^\chi)^\Gamma = (X (F_\infty)^\chi_{\mathrm{fin}})^\Gamma$.
The following isomorphisms and exact sequences play important roles in this section.
\begin{equation}\label{eq_pre1}
\mathfrak{X} (F_\infty)^\chi_\Gamma \cong 
\Gal (M (F_0)/F_0)^\chi \quad \text{and} \quad 
X (F_\infty)^\chi_\Gamma \cong A (F_0)^\chi.
\end{equation}
\begin{equation}\label{eq_pre2}
0 \to (\mathcal{U}^1/ \mathcal{E}^1)^\chi \to  \Gal (M (F_0)/F_0)^\chi \to A (F_0)^\chi \to 0.
\end{equation}
\begin{equation}\label{eq_pre3}
0 \to (X (F_\infty)^\chi)^\Gamma \to
\Gal (M (F_\infty) / L (F_\infty))^\chi_{\Gamma} \to \mathfrak{X} (F_\infty)^\chi_\Gamma \to 
X (F_\infty)^\chi_\Gamma \to 0.
\end{equation}
For the results given in this paragraph, note that similar results 
hold for the case of the cyclotomic $\mathbb{Z}_p$-extension 
of real abelian fields, and one can also show our results quite similarly. 
See, e.g., \cite{Oza95}, \cite{O-T95}, \cite{Oza97J}, \cite{Oza97T}, \cite{A-NQD}, \cite{B-NQD}, \cite{Fuj}.

\begin{rem}\label{rem_Maire}
Note that $L(F_0) F_\infty /F_\infty$ is the maximal unramified subextension of 
$M (F_0)/ F_\infty$ in our situation, and $\Gal (M (F_0)/L(F_0) F_\infty)$ is
isomorphic to the $\mathbb{Z}_p$-torsion subgroup of $\mathcal{U}^1/\mathcal{E}^1$
(this can be shown by using the same argument given in \cite[Section 4]{Fuj}).
Hence, for the question on the non-triviality of $\Gal (M (F_\infty )/L(F_\infty))$, 
it seems significant to study the $\mathbb{Z}_p$-torsion subgroup of $\mathcal{U}^1/\mathcal{E}^1$ 
(more generally, a similar object of $F_n$).
Christian Maire gave a remark on the earlier studies on 
the structure of the $\mathbb{Z}_p$-torsion subgroup of the ``group of (semi) local units modulo 
the completion of the group of global units''.
In particular, studying analogous objects of the ``Kummer-Leopoldt constant'' and 
the ``$p$-adic normalized regulator'' (see \cite{A-NQD}, \cite{Gras}) may be useful.
See also Appendix \ref{sec_appendix}.
\end{rem}

\subsection{Criteria for (Q2)}
\begin{lem}[cf.~e.g., \cite{Kra}, \cite{O-T95}]\label{lem_U/Enot1}
Assume that $K$, $p$, $E$ satisfy (C1), (C2), (C3).
If $(\mathcal{U}^1 / \mathcal{E}^1)^\chi$ is not trivial, 
then $\Gal (M (F_\infty) / L (F_\infty))^\chi$ is not trivial.
\end{lem}

\begin{proof}
This proposition can be obtained by using the same argument given in the 
proof of \cite[Lemma 2]{O-T95}, 
which treats the case of the cyclotomic $\mathbb{Z}_p$-extension of real abelian fields.
(See also \cite[Theorem 3]{Kra}.)
In fact, by using (\ref{eq_pre1}), (\ref{eq_pre2}), (\ref{eq_pre3}), 
we can see that the triviality of $\Gal (M (F_\infty) / L (F_\infty))^\chi$ 
implies the triviality of $(\mathcal{U}^1 / \mathcal{E}^1)^\chi$.
\end{proof}

\begin{prop}\label{prop_Q2}
Assume that $K$, $p$, $E$ satisfy (C1), (C2), (C3).
\begin{itemize}
\item[(i)] If $\mathcal{E}^1$ is contained in $\mathcal{U}^2$, then 
$\Gal (M (F_\infty) / L (F_\infty))^\chi$ is not trivial.
\item[(ii)] If $\mathcal{U}^1$ contains a primitive $p$th root of unity, then 
$\Gal (M (F_\infty) / L (F_\infty))^\chi$ is not trivial.
\end{itemize}
\end{prop}

\begin{proof}
There is a $\mathbb{Z}_p [\Delta]$-module isomorphism 
\begin{equation}\label{principal_units}
E [\pi] \cong \mathcal{U}^1 / \mathcal{U}^2.
\end{equation}
(See \cite[Lemma 10.4]{Rub99}. Note that it was assumed that $p>7$ 
at \cite[Section 10]{Rub99}, however, 
we can show that this assertion holds for $p \geq 3$.
See also \cite[Lemma 9]{C-W}.)
Then, (i) follows from this isomorphism and Lemma \ref{lem_U/Enot1}.

We shall prove (ii). 
Assume that $\mathcal{U}^1$ contains a primitive $p$th root of unity $\zeta_p$.
That is, $(F_0)_{\mathfrak{P}}$ is isomorphic to $\mathbb{Q}_p (\zeta_p)$ 
(see also the proof of \cite[Lemma 12]{C-W}).
Since $\zeta_p \, \mathcal{U}^2$ generates $\mathcal{U}^1 / \mathcal{U}^2$, 
it follows that $\zeta_p$ is contained in $(\mathcal{U}^1)^\chi$.
We claim that $\mathcal{E}^1$ does not contain $\zeta_p$.
Note that the global field $F_0$ does not contain 
a primitive $p$th root of unity. 
(If it contains, then both primes of $K$ lying above $p$ ramifies.
However, it cannot be occurred because $E$ has good reduction at $p$.
See, e.g., \cite[Corollary 3.17]{Rub99}.)
By combining this fact and the validity of the 
$\{ \mathfrak{P} \}$-adic analog of Leopoldt's conjecture, the claim can be shown 
(cf.~also, e.g., \cite[Lemma 3.1 and Corollary 3.2]{Gras}).
By this claim, we see that $(\mathcal{U}^1 / \mathcal{E}^1)^\chi$ is not trivial.
\end{proof}

\begin{rem}\label{rem_unit}
Assume that $\mathcal{U}^1$ does not contain any primitive $p$th root of unity.
In this case, $(\mathcal{U}^1)^\chi$ is a free $\mathbb{Z}_p$-module of rank $1$. 
By using this fact and (\ref{principal_units}), 
we can see that $(\mathcal{U}^1 / \mathcal{E}^1)^\chi$ is trivial 
if and only if there is a (global) unit $u$ of $F_0$ such that 
$u^{p-1} \not\equiv 1 \pmod{\mathfrak{P}^2}$.
\end{rem}

\begin{rem}\label{rem_Fp}
Let $\widetilde{E} (\mathbb{F}_p)$ be the group of 
$\mathbb{F}_p$-rational points of the reduction of $E$ at $p$.
Assume that $|\widetilde{E} (\mathbb{F}_p)|$ is divisible by $p$.
Then we can see that 
$\psi (\mathfrak{p}) + \overline{\psi (\mathfrak{p})} \equiv 1 \pmod{p}$,
where $\psi$ is the Gr{\"o}ssencharacter of $E$ over $K$
(see, e.g., \cite[Chapter II, Corollary 10.4.1 (b)]{Sil2}).
By using the argument given in the proof of \cite[Lemma 12]{C-W}, 
we see that $(F_0)_{\mathfrak{P}}$ contains a primitive $p$-th root of unity.
Hence (Q2) has an affirmative answer by Proposition \ref{prop_Q2} (ii).
\end{rem}

\begin{rem}\label{rem_old}
Let $L(E/\mathbb{Q}, s)$ (resp. $L(E/K, s)$) be the complex $L$-function of $E$ over $\mathbb{Q}$
(resp. over $K$). 
We assume that $L (E/\mathbb{Q}, 1) \neq 0$. 
In this situation, we can show that if the $p$-rank of $A (F_0)^\chi$ is odd 
then (Q2) has an affirmative answer.
We give an outline of the proof.
We first note that $L(E/K, 1)$ is also not equal to $0$, and then 
$E (K)$ is finite (see, e.g., \cite{C-W}, \cite{Mil}).
Let $S_{\pi} (E/K) (\subset H^1 (\Gal (\overline{K}/K), E [\pi]))$ 
be the Selmer group relative to $\pi$, 
and $S'_{\pi} (E/K)$ the enlarged Selmer group relative to $\pi$ 
(see, e.g., \cite[p.32]{Per}, \cite[Definition 6.3]{Rub99}).
We may assume that $\widetilde{E} (\mathbb{F}_p) \not\equiv 0 \pmod{p}$ 
(see Remark \ref{rem_Fp}).
Under this assumption, we can show that $S'_{\pi} (E/K) \cong S_{\pi} (E/K)$ 
(see also \cite[p.35]{Per}).
Note that 
\[ S'_{\pi} (E/K) \cong \Hom (\Gal (M (F_0)/ F_0)^\chi, E [\pi]). \]
(See \cite[Theorem 6.5]{Rub99}. 
In our situation, this holds even when $p=3$.)
We claim that the $p$-rank of $S'_{\pi} (E/K)$ is even.
Let $\Sha (E/K)$ (resp. $\Sha (E/\mathbb{Q})$) be the Tate-Shafarevich group of $E/K$
(resp. $E/\mathbb{Q}$).
We denote by $\Sha (E/K) [\pi]$ the $\pi$-torsion subgroup of $\Sha (E/K)$ 
(we also define $\Sha (E/K) [p]$, $\Sha (E/K) [\overline{\pi}]$, and 
$\Sha (E/\mathbb{Q}) [p]$ similarly).
In our situation, 
it is known that both $|\Sha (E/K)|$ and $|\Sha (E/\mathbb{Q})|$ are finite (Rubin \cite{Rub87}).
Then, by the result of Cassels (see, e.g., \cite[Chapter X, Theorem 4.14]{Sil}), 
the $p$-rank of $\Sha (E/\mathbb{Q})$ is even.
Moreover, we can show that $S_{\pi} (E/K) \cong \Sha (E/K) [\pi]$ in our situation.
We write $K = \mathbb{Q} (\sqrt{d})$ with a negative square-free integer $d$.
Let $E^d$ be the quadratic twist of $E$ by $d$. 
We can obtain the following:
\[ \Sha (E/K) [p] \cong \Sha (E/\mathbb{Q}) [p] \oplus \Sha (E^d/\mathbb{Q}) [p] \]
(see, e.g., \cite[Lemma 3.1]{O-P}, the argument given in \cite{Mil}), 
\[ \Sha (E / \mathbb{Q}) [p] \cong \Sha (E^d/\mathbb{Q}) [p] \]
(this was suggested by an anonymous referee of an earlier manuscript,
and the author express his thanks to him/her),
\[ \Sha (E/K) [p] \cong \Sha (E/K) [\pi] \oplus \Sha (E/K) [\overline{\pi}], \quad 
|\Sha (E/K) [\pi]| = |\Sha (E/K) [\overline{\pi}]| \]
(cf. the argument given in \cite[p.260]{Gre83}).
By using these results, we see that the $p$-rank of $\Sha (E/K) [\pi]$ is even.
The claim follows, and 
hence if the $p$-rank of $A (F_0)^\chi$ is odd, then $A (F_0)^\chi$ is not isomorphic to 
$\Gal (M (F_0)/F_0)^\chi$ (and $(\mathcal{U}^1 / \mathcal{E}^1)^\chi$ is not trivial).
\end{rem}

\subsection{Criteria for (Q1)}

\begin{prop}\label{prop_U/E=1}
Assume that $K$, $p$, $E$ satisfy (C1), (C2), (C3).
If $(\mathcal{U}^1 / \mathcal{E}^1)^\chi$ is trivial and $\rank E(\mathbb{Q}) \geq 1$, 
then $X (F_\infty)_{\mathrm{fin}}^{\chi}$ is not trivial.
\end{prop}

\begin{proof}
We first recall that $\Gal (M(F_\infty)/L(F_\infty))^\chi$ is not trivial 
by Theorem Z (ii).

The essential idea of the following argument was given to the author
by Satoshi Fujii (concerning his work \cite[Section 4]{Fuj}). 
(Note that the same idea also can be found in \cite[Th{\'e}or{\`e}me 2.1]{B-NQD}. 
See also \cite[Proposition 4]{A-NQD}.)
Since $(\mathcal{U}^1 / \mathcal{E}^1)^\chi$ is trivial, we see that 
$\mathfrak{X} (F_\infty)^\chi_\Gamma \cong X (F_\infty)^\chi_\Gamma$ 
by using (\ref{eq_pre1}) and (\ref{eq_pre2}).
Recall also that $(\mathfrak{X} (F_\infty)^\chi)^\Gamma$ is trivial.
From these facts and (\ref{eq_pre3}), we see that 
\[ (X (F_\infty)^\chi)^\Gamma \cong \Gal (M(F_\infty)/L(F_\infty))^\chi_\Gamma. \]
Since $\Gal (M(F_\infty)/L(F_\infty))^\chi$ is not trivial, 
we can show that $\Gal (M(F_\infty)/L(F_\infty))^\chi_\Gamma$ is not trivial 
by using Nakayama's lemma.  
Then, $(X (F_\infty)^\chi)^\Gamma = (X (F_\infty)_{\mathrm{fin}}^\chi)^\Gamma$ 
is not trivial.
The assertion has been shown.
\end{proof}

\begin{cor}\label{cor_finite}
Assume that $K$, $p$, $E$ satisfy (C1), (C2), (C3).
If $(\mathcal{U}^1 / \mathcal{E}^1)^\chi$ is trivial, $\rank E(\mathbb{Q}) = 1$,
and $|A(F_0)^\chi| =p$, 
then $X (F_\infty)^{\chi}$ is non-trivial and finite.
\end{cor}

\begin{proof}
When $|A(F_0)^\chi| =p$, we see that 
$X (F_\infty)_{\mathrm{fin}}^\chi$ is trivial or 
$X (F_\infty)^\chi = X (F_\infty)_{\mathrm{fin}}^\chi$ 
(see the proof of \cite[Theorem 2]{Oza97J}).
By Proposition \ref{prop_U/E=1}, we 
see that the former case never occurs 
under the assumption of this corollary.
\end{proof}

By using the argument given in the above proof, we can see that 
if $|A(F_0)^\chi| =p$ and $X (F_\infty)^\chi$ is not finite, 
then $X (F_\infty)_{\mathrm{fin}}^\chi$ is trivial
(cf. \cite[Corollary 2.2]{Itoh18}).
Moreover, we can also show the following result
(cf. Sections 1--4 of \cite{Itoh18}).

Let $\kappa$ be the restriction of $\phi$ on $\Gamma$ 
(see Section \ref{sub_questions}).
Put $r = \rank E(\mathbb{Q})$.
It is known that the characteristic ideal 
$\mathrm{Char}_{\Lambda} X (F_\infty)^\chi$ is contained in 
$(T +1 - \kappa (\gamma_0))^{r-1} \Lambda$ 
(see \cite{Gre01}, \cite{F-K}).

\begin{prop}\label{prop_noNTFS}
Let the notation be as above.
Assume that $K$, $p$, $E$ satisfy (C1), (C2), (C3).
If $r \geq 2$ and $|A (F_0)^\chi|=p^{r-1}$, 
then $X (F_\infty)_{\mathrm{fin}}^\chi$ is trivial and
$\Char X (F_\infty)^\chi = (T +1 - \kappa (\gamma_0))^{r-1} \Lambda$.
\end{prop}

\begin{proof}
Let $f (T)$ be a generator of $\Char X (F_\infty)^\chi$.
It is well known that
\[ \dfrac{|(X (F_\infty)^\chi)^\Gamma|}{|X (F_\infty)^\chi_\Gamma|} =
|f (0)|_p, \]
where $| \cdot |_p$ denotes the normalized $p$-adic (multiplicative)
absolute value (see, e.g., \cite[Exercise 13.12]{Was}).

Recall that $X (F_\infty)^\chi_\Gamma \cong A (F_0)^\chi$ and
$(X (F_\infty)^\chi)^\Gamma = (X (F_\infty)_{\mathrm{fin}}^\chi)^\Gamma$.
As noted above, $f (T)$ is divisible by $(T+1- \kappa(\gamma_0))^{r-1}$.
Hence, if $| A(F_0)^\chi | =p^{r-1}$, it must be satisfied that
$(X (F_\infty)_{\mathrm{fin}}^\chi)^\Gamma$ is trivial and
$\Char X (F_\infty)^\chi = (T+1- \kappa (\gamma_0))^{r-1} \Lambda$.
Note that the triviality of $(X (F_\infty)_{\mathrm{fin}}^\chi)^\Gamma$ implies the 
triviality of $X (F_\infty)_{\mathrm{fin}}^\chi$. 
\end{proof}

\begin{rem}
If $E$ and $p$ satisfy the assumption of the above Proposition \ref{prop_noNTFS},
then Conjecture 1.2 of \cite{F-K} holds for $E$ and $p$. 
However, we mention that this does not imply the validity of 
Conjecture 1.1 of \cite{F-K}.
\end{rem}

\section{Examples for the questions (Q1) and (Q2)}\label{sec_examples}

\subsection{Software used in the example calculation}
The author used PARI/GP \cite{PARI} (formerly 2.11.2 and finally version 2.13.1) mainly to 
compute the ideal class groups, units, values of $L$-functions, etc 
(all examples are computed (or recomputed) by using version 2.13.1).
However, for the computation of the rank of elliptic curves, the author used Sage \cite{Sage} 
(\texttt{mwrank} \cite{mwrank} was mainly used). 
In the computation on Sage, the article \cite{Kimura} was very helpful.
The author also would like to express thanks to Iwao Kimura for giving comments.

\subsection{Fukuda-Komatsu's examples}\label{sec_F-K}
In this subsection, we put $K = \mathbb{Q} (\sqrt{-1})$ and $p=5$.
We can find a negative example for (Q1) in Fukuda-Komatsu's paper \cite{F-K}.

\begin{example}[see Fukuda-Komatsu \cite{F-K}]\label{example_negativeQ1}
Let $E$ be an elliptic curve defined by the Weierstrass equation 
\[ y^2 = x^3 + 99x. \]
Then $K$, $p$, $E$ satisfy (C1), (C2), (C3).
This case is treated in \cite[Section 4.1]{F-K}.
Note that $\rank E(\mathbb{Q})$ is 2 
(\cite[Table des valeurs des $\lambda (l_{\mathfrak{p},i}^*)$: I]{BGS}), 
and then $X (F_\infty)^\chi$ is not finite by Theorem Z (i).
It is also known that $|A (F_0)|=5$, hence the 
infiniteness of $X (F_\infty)^\chi$ implies that $|A (F_0)^\chi|=5$.
By Proposition \ref{prop_noNTFS}, 
we see that $X (F_\infty)_{\mathrm{fin}}^\chi$ is trivial.
Then, this is a negative example for (Q1).
On the other hand, (Q2) has an affirmative answer for this case by Theorem Z (ii). 
Hence the assertion of (Q2) is actually weaker than that of (Q1).
Note that Proposition \ref{prop_noNTFS} also gives an alternative proof of the 
fact (already confirmed in \cite{F-K}) that 
$\Char X (F_\infty)^\chi = (T +1 - \kappa (\gamma_0)) \Lambda$.
\end{example}

In Sections 4.2 and 4.3 of \cite{F-K}, the examples such that 
$X (F_\infty)^\chi$ is finite are also given.
We shall introduce some of them.

\begin{example}[see Fukuda-Komatsu \cite{F-K}]\label{example_fin1}
Let $E$ be an elliptic curve defined by the Weierstrass equation 
\[ y^2 = x^3 + 1331x. \]
Then $K$, $p$, $E$ satisfy (C1), (C2), (C3).
It is stated in \cite[Section 4.2]{F-K} that $|A(F_0)| =5$ and $X (F_\infty)^\chi$ is finite.
Note that it is not explicitly stated that $|A(F_0)^\chi| =5$ in \cite{F-K}.
(Although it seems that they had obtained this fact, 
the author also confirmed this fact.)
Hence we see that $X (F_\infty)^\chi$ is non-trivial and finite.
Similarly, it is also stated in \cite[Section 4.2]{F-K} that 
\[ y^2 = x^3 + 2197x \]
is also an example such that $X (F_\infty)^\chi$ is finite.
For this case, it can be also checked that $|A (F_0)^\chi| = 5$, and hence 
$X (F_\infty)^\chi$ is non-trivial.
\end{example}

\begin{rem}
In the computation concerning Example \ref{example_fin1} (and below Example \ref{sub_307}), 
the author used an explicit Kummer generator of $F_0$ over $K$ written in \cite[p.547]{F-K} 
to obtain a defining polynomial of $F_0$.
For the curves given in Example \ref{example_fin1}, the author checked that 
$|A (F_0)^\chi| = 5$ by using the following two ways.
One is computing the $\Delta$-action for a generator of $A (F_0)$.
The other is comparing the information on the $\chi^i$-part of $\mathfrak{X} (F_\infty)$ 
given in \cite[Table des valeurs des $\lambda (l_{\mathfrak{p},i}^*)$: I]{BGS} 
and ideal class group of the quadratic subextension of $F_0/K$.
\end{rem}

In their computation, 
Fukuda-Komatsu \cite{F-K} used the $p$-adic $L$-function and the 
``Ichimura-Sumida type'' criterion for elliptic units to determine 
the characteristic polynomials.
Their method is also a powerful tool to confirm our questions.
For example, our Corollary \ref{cor_finite} is not applicable for the cases of  
$y^2 = x^3 + 1331x$ and $y^2 = x^3 + 2197x$.
However, the ``Ichimura-Sumida type'' criterion 
seems to need the information on the elliptic units of 
higher layers of $F_\infty/ F_0$ in general. 
Our criteria only need the information on $F_0$, which can be easily 
computed by using existing software (at least when $p=3$).
Hence, our criteria seem suitable to confirm various examples 
(see Section \ref{sub_11}).

We also note that if an explicit generator $f(T)$ of $\Char X (F_\infty)^\chi$ is known, 
we can check whether $X (F_\infty)_{\mathrm{fin}}^\chi$ is trivial or not by 
comparing $|A (F_0)^\chi|$ and $|f(0)|_p$. 
(See the proof of Proposition \ref{prop_noNTFS}. 
See also \cite{Itoh18}.)

\subsection{Example for (Q2) with $K = \mathbb{Q} (\sqrt{-1})$ and $p=5$}\label{sub_307}
Here we give an example for (Q2) such that $\rank E( \mathbb{Q}) =0$ and 
$X (F_\infty)^\chi$ is non-trivial. 

\begin{example}\label{ex_307}
We put $K = \mathbb{Q} (\sqrt{-1})$ and $p=5$.
Let $E$ be an elliptic curve defined by the Weierstrass equation 
\[ y^2 = x^3 - 307^2 x. \]
In this case, it is known that $L(E/\mathbb{Q},1) \neq 0$ (see \cite[Theorem 1]{Raz}).
The author checked that $|A(F_0)^\chi|=5$.
Then, the criterion given in Remark \ref{rem_old} is applicable, 
and hence this is a non-trivial example such that (Q2) has an affirmative answer.
We also remark that Proposition \ref{prop_Q2} (ii) is not applicable for this example.
However, the author also checked the non-triviality of $(\mathcal{U}^1 / \mathcal{E}^1)^\chi$  
by using a more direct method (Remark \ref{rem_unit}).
\end{example}

\begin{rem}
For the above example, the order of $A(F_0)$ is $5$.
Although the fact $|A (F_0)^\chi| = 5$ can be confirmed by observing the $\Delta$-action, 
we can also check this by using the following way.
In this situation, we can see that if $|\Sha (E/\mathbb{Q})|$ is divisible by $5$, then 
$A (F_0)^\chi$ is not trivial. 
(See Remark \ref{rem_old}. See also the proof of \cite[Corollary 6.10]{Rub99}.)
We also note that the full Birch and Swinnerton-Dyer conjecture holds for $E$ 
(see \cite[p.26, Theorem]{Rub91m}).
Then, by computing the analytic order of $\Sha (E/\mathbb{Q})$, 
we can confirm that $A (F_0)^\chi$ is not trivial, and hence $|A (F_0)^\chi| = 5$.
\end{rem}

\subsection{Examples with $K = \mathbb{Q} (\sqrt{-11})$ and $p=3$}\label{sub_11}
Let $E_\circ^d$ be an elliptic curve defined by the Weierstrass equation 
\[ y^2 = x^3 - 264 d^2 x + 1694 d^3, \]
where $d$ is a non-zero square-free integer.
We put $K = \mathbb{Q} (\sqrt{-11})$ and $p=3$.
It is well known that $E_\circ^d$ has complex multiplication by $O_K$ 
(see, e.g., \cite{Hada}).
Note also that $E_\circ^d$ has good reduction at $p=3$ if and only if $d \equiv 0 \pmod{3}$ 
(see \cite{Hada}).
We also note that $E_\circ^d$ and $E_\circ^{-11d}$ are isomorphic over $K$.
Hence, in the remaining part of this subsection, 
we assume that $d$ satisfies the following condition.
\begin{itemize}
\item[(D1)] $d$ is a square-free integer satisfying $d \equiv 0 \pmod{3}$ and 
$d \not\equiv 0 \pmod{11}$.
\end{itemize}
Then, under (D1), $K$, $p$, $E_\circ^d$ satisfy (C1), (C2), (C3).
We choose $\mathfrak{p}$ as a prime generated by $(-1-\sqrt{-11})/2$.
We put $d' = d/3$, then 
we can see that 
\[ F_0 = K \left(\sqrt{d' \, (11- \sqrt{-11})} \right). \]
This can be obtained by using an explicit endomorphism given in \cite[Theorem 3]{Raj}.
(However, it seems that the multiplication by $(-1+\sqrt{-11})/2$ 
endomorphism given in \cite[Theorem 3]{Raj} is actually the  
multiplication by $(-1-\sqrt{-11})/2$ endomorphism.)

Let $\overline{\mathfrak{p}}$ be the conjugate of $\mathfrak{p}$.
Then, $\overline{\mathfrak{p}}$ is unramified in $F_0$.
Moreover, we can see that $\overline{\mathfrak{p}}$ splits completely in $F_0$ 
if and only if $d' \equiv 1 \pmod{3}$ (i.e., $d \equiv 3 \pmod{9}$).
We also note that $\mathcal{U}^1$ contains a primitive third root of unity 
if and only if $\overline{\mathfrak{p}}$ splits completely in $F_0$.
Hence, by Proposition \ref{prop_Q2} (ii), we have obtained the following result.
\begin{itemize}
\item If $d \equiv 3 \pmod{9}$, then (Q2) has an affirmative answer for $E_\circ^d$.
\end{itemize}

Let $L (E_\circ^d /\mathbb{Q}, s)$ be the complex $L$-function of 
$E_\circ^d$ over $\mathbb{Q}$.
We also note that the root number of $E_\circ^d$ is $-1$ when $d >0$. 
(See \cite[Theorem 19.1.1]{Gross}. Recall also that $d$ is assumed to be prime to $11$.)
Then, $L (E_\circ^d /\mathbb{Q}, 1)=0$ for this case.
Hence if $d>0$ and 
the Birch and Swinnerton-Dyer conjecture (or the parity conjecture) holds for $E_\circ^d$, 
we see that $\rank E_\circ^d (\mathbb{Q}) \geq 1$ (and (Q2) has an affirmative 
answer by Theorem Z (ii)).

We shall give several examples for the case when $d \equiv 6 \pmod{9}$.
First, we shall consider (Q2).
For this question, we can use Theorem Z (ii) and Proposition \ref{prop_Q2}.
Recall that $\mathcal{U}^1$ does not contain a primitive third root of unity 
when $d \equiv 6 \pmod{9}$. 
Hence we can check whether $(\mathcal{U}^1 /\mathcal{E}^1)^\chi$ is trivial or not 
by using the method stated in Remark \ref{rem_unit}.

\begin{example}\label{ex_11Q2p}
Assume that $d>0$ and $d \equiv 6 \pmod{9}$.
In the range $1 < d < 3000$, the following values satisfy 
that $|A (F_0)^\chi| \neq 1$.
(We note that $A (F_0) = A (F_0)^\chi$ in this situation.
Hence, the process of extracting the $\chi$-part is not needed.)
\begin{equation}\label{eq_11Q2p} 
\begin{array}{rl}
d = & 78, 87, 141, 177, 186, 195, 213, 285, 357, 366, \\
  & 393, 447, 501, 510, 537, 609, 699, 717, 753, 807, \\
  & 843, 861, 870, 915, 942, 969, 987, 1005, 1149, 1167, \\
  & 1203, 1230, 1293, 1365, 1374, 1482, 1545, 1554, 1635, 1662, \\
  & 1689, 1707, 1779, 1842, 1851, 1887, 1923, 1959, 2085, 2121, \\
  & 2139, 2202, 2247, 2301, 2346, 2454, 2463, 2481, 2490, 2562, \\
  & 2571, 2589, 2634, 2679, 2715, 2769, 2877, 2922, 2949, 2967, 2985
\end{array}
\end{equation}
Recall that $L (E_\circ^d /\mathbb{Q}, 1)=0$ in this situation.
Hence, if $L' (E_\circ^d /\mathbb{Q}, 1) \neq 0$, we see that $\rank E_\circ^d (\mathbb{Q}) = 1$ 
(\cite[Corollary C]{Rub87}).
In the above values, the condition $L' (E_\circ^d /\mathbb{Q}, 1) \neq 0$ is satisfied 
except for the cases when $d= 141, 807, 2121$.
Moreover, for all of these three values, the author checked that 
$\rank E_\circ^d (\mathbb{Q}) = 3$.
Hence, for the values of $d$ listed above, (Q2) has an affirmative answer 
by Theorem Z (ii).
(For $d= 141, 807, 2121$, 
it can be checked that $(\mathcal{U}^1 /\mathcal{E}^1)^\chi$ is not trivial.
Hence, Proposition \ref{prop_Q2} is also applicable for these three values.)
\end{example}

\begin{example}\label{ex_11Q2n}
Assume that $d < 0$ and $d \equiv 6 \pmod{9}$.
In the range $-3000 < d <0$, the following $48$ values of $d$ satisfy that 
$|A (F_0)^\chi| \neq 1$.
\begin{equation}\label{eq_11Q2n} 
\begin{array}{rl}
d = & -2955, -2910, -2874, -2847, -2757, -2730, -2703, -2649, -2613, -2559, \\
    & -2514, -2478, -2469, -2433, -2361, -2298, -2271, -2262, -2154, -2109, \\ 
    &  -2010, -1974, -1965, -1758, -1731, -1695, -1623, -1461, -1281, -1263, \\
    &  -1227, -1137, -1119, -1110, -1065, -1038, -1002, -993, -678, -651, \\
    & -489, -399, -390, -327, -255, -174, -93, -21. 
\end{array}
\end{equation}
We can see that the $45$ values except for $-2910, -2361, -1731$ satisfy 
that $(\mathcal{U}^1 /\mathcal{E}^1)^\chi$ is not trivial, 
then (Q2) has an affirmative answer for these $45$ values.
We note that $L (E_\circ^d /\mathbb{Q},1)$ is approximately $0$ for several values in the above list.
(Since the root number of $E_\circ^d$ is $+1$ in this case (\cite[Theorem 19.1.1]{Gross}), 
if $L (E_\circ^d /\mathbb{Q},1) = 0$ then $\rank E_\circ^d (\mathbb{Q}) \geq 2$ under 
the parity conjecture.)
See the following Examples \ref{ex_NTFS} and \ref{ex2361}. 
(For the case when $d=-2361$, we will later see that (Q1) has an affirmative answer, 
and hence (Q2) also has an affirmative answer.)
\end{example}

Next, we shall consider (Q1).

\begin{example}\label{ex_11Q1}
We shall back to the situation treated in Example \ref{ex_11Q2p}.
Assume that $d>0$ and $d \equiv 6 \pmod{9}$.
For the values given in (\ref{eq_11Q2p}), it was checked that 
$\rank (E_\circ^d) \geq 1$.
Moreover, if $(\mathcal{U}^1 /\mathcal{E}^1)^\chi$ is trivial  
then (Q1) has an affirmative answer by Proposition \ref{prop_U/E=1}.
For the values stated in (\ref{eq_11Q2p}), the following values satisfy that 
$(\mathcal{U}^1 /\mathcal{E}^1)^\chi$ is trivial.
\[ \begin{array}{rl}
 d = & 78, 87, 186, 195, 213, 285, 393, 447, 501, 510, 537, 609, 699, 717, \\
     & 753, 861, 870, 915, 969, 987, 1005, 1167, 1230, 1293, 1365, 1482, \\
     & 1545, 1635, 1662, 1707, 1779, 1842, 1851, 1887, 1923, 1959, 2085, \\
     & 2139, 2247, 2454, 2463, 2481, 2562, 2571, 2634, 2679, 2715, 2769, \\
     & 2877, 2922, 2967, 2985.
\end{array} \]
Note that all of these values satisfy $\rank (E_\circ^d) = 1$.
In addition, if $|A (F_0)^\chi| = 3$, then $X (F_\infty)^\chi$ is 
non-trivial and finite by Corollary \ref{cor_finite}.
In the above list, the condition $|A (F_0)^\chi| = 3$ is satisfied 
except for the cases when $d= 1167, 1482, 2247$.
\end{example}

The above are the affirmative examples for (Q1).
We can also find many negative examples (that is, 
$X (F_\infty)^\chi$ is infinite and 
$X (F_\infty)_{\mathrm{fin}}^{\chi}$ is trivial).

\begin{example}\label{ex_NTFS}
As noted in Example \ref{ex_11Q2n}, several values of $d$ stated in (\ref{eq_11Q2n}) satisfy that 
$L (E_\circ^d /\mathbb{Q},1)$ is approximately $0$.
Such values are the following: 
\[ \begin{array}{rl}
d = & -2874, -2847, -2730, -2703, -2649, -2514, -2361, -2271, -2154, -1974, \\
  & -1965, -1758, -1119, -1002, -651, -489, -399, -390, -255, -174, -21.
\end{array} \]
The author checked that $\rank E_\circ^d (\mathbb{Q}) = 2$ for all of the above values.
Hence, for these values, we see that $X (F_\infty)^\chi$ is infinite by Theorem Z (i).
Moreover, if $|A(F_0)^\chi| = 3$, then $X (F_\infty)_{\mathrm{fin}}^{\chi}$ is trivial 
(and $\Char X (F_\infty)^{\chi} = (T+1 -\kappa (\gamma_0)) \Lambda$) 
by Proposition \ref{prop_noNTFS}.
In the above list, we can check that $|A(F_0)^\chi| = 3$ except for the cases when 
$d=-2703, -2361$.
\end{example}

\begin{example}\label{ex_NTFSr3}
We shall consider the cases when $d= 141, 807, 2121$.
Recall that $\rank E_\circ^d (\mathbb{Q}) = 3$ for these values
(see Example \ref{ex_11Q2p}).
Moreover, it can be checked that $|A(F_0)^\chi| = 9$ for all of these values. 
Then we see that 
$X (F_\infty)_{\mathrm{fin}}^{\chi}$ is trivial and 
$\Char X (F_\infty)^{\chi} = (T+1 -\kappa (\gamma_0))^2 \Lambda$ 
by Proposition \ref{prop_noNTFS}.
\end{example}

The above Example \ref{ex_NTFSr3} gives examples of rank $3$ elliptic curves 
such that Conjecture 1.2 of \cite{F-K} is valid.

\begin{rem}
We can also find examples satisfying $X (F_\infty)^\chi$ is infinite and 
$X (F_\infty)_{\mathrm{fin}}^{\chi}$ is trivial for the case when 
$d \equiv 3 \pmod{9}$.
For instance, $d=-159,-114,-51$ are such values.
\end{rem}

The following is an example such that $X (F_\infty)^\chi$ is infinite and 
$X (F_\infty)_{\mathrm{fin}}^{\chi}$ is not trivial.

\begin{example}\label{ex2361}
We shall consider the case when $d=-2361$.
Recall that $\rank  E^{-2361}_\circ (\mathbb{Q}) = 2$, and hence $X (F_\infty)^\chi$ is infinite 
(Example \ref{ex_NTFS}).
In this case, we also see that $(\mathcal{U}^1 /\mathcal{E}^1)^\chi$ is trivial 
(Example \ref{ex_11Q2n}).
Thus, by Proposition \ref{prop_U/E=1}, 
we see that $X (F_\infty)_{\mathrm{fin}}^{\chi}$ is not trivial.
\end{example}

As a conclusion of this subsection, 
for the case of $E_\circ^d$ with an integer $d$ satisfying (D1), 
we have confirmed the following:
\begin{itemize}
\item In the range $-3000 < d < 3000$, (Q2) has an affirmative answer 
except for the cases when $d= -2910, -1731$.
\item Similar to the situation treated in Section \ref{sec_F-K}, 
both affirmative and negative examples exist for (Q1).
\end{itemize}

\appendix

\section{Similar questions for the case when $p=2$}\label{sec_appendix}

\subsection{Questions and results}\label{sub_results}
In this Appendix B, we fix $K = \mathbb{Q} (\sqrt{-7})$ and $p=2$.
For a non-zero square free integer $d$, let $E_*^d$ be the elliptic curve over $\mathbb{Q}$
defined by the following equation
\[ y^2 = x^3 + 21 d x^2 + 112 d^2 x. \]
It is well known that $E_*^d$ has complex multiplication by $O_K$.
This situation is well studied, and we shall recall several facts.
Note that $E_*^d$ has good reduction at $2$ if and only if $d \equiv 1 \pmod{4}$ 
(see, e.g., \cite{Hada}).
Moreover, if $d$ is prime to $7$, then $E_*^d$ is isomorphic to $E_*^{-7d}$ over $K$ 
(see, e.g., \cite[Chapter X]{Sil}, \cite[Section 7]{Gon}, \cite[Section 2]{Mil}).
Hence, it is sufficient to consider $E_*^d$ such that $d$ satisfies the following condition. 
\begin{itemize}
\item[(D2)] $d$ is a square-free integer satisfying
$d \equiv 1 \pmod{4}$ and $d \not\equiv 0 \pmod{7}$.
\end{itemize} 
(See also, e.g., \cite[Section 7]{Gon}, \cite{C-C}, \cite{CLT}.)
Note that $p=2$ splits in $K$, and the class number of $K$ is $1$.
Let $\mathfrak{p}$ be a prime of $K$ lying above $2$. 
We put $\pi = \psi (\mathfrak{p})$ 
(where $\psi$ is the Gr{\"o}ssencharacter of $E_*^{d}$ over $K$), 
and $F_n = K (E_*^d [\pi^{n+2}])$ for all $n \geq 0$ 
(this definition is slightly different from the case when $p \geq 3$).
Note that $F_0/ K$ is a quadratic extension (see, e.g., \cite[Section 2]{Gon}, \cite{C-C}).
It is known that $F_n /K$ is totally ramified at $\mathfrak{p}$ 
(see, e.g., \cite[(3.6)Proposition (i)]{Rub87}, \cite{C-C}).
We denote by $\mathfrak{P}$ the unique prime of $F_0$ lying above $\mathfrak{p}$.
We put $F_\infty = \bigcup_n F_n$, then $F_\infty /F_0$ is a $\mathbb{Z}_2$-extension
unramified outside $\mathfrak{P}$.

In the following, we choose $\mathfrak{p}$ as a prime generated by $(-1-\sqrt{-7})/2$.
It is known that 
\begin{equation}\label{eq_F0p2}
F_0 = K (\sqrt{d \sqrt{-7}}). 
\end{equation}
(See \cite[Lemma 2.2]{C-C}, however, notice the difference of the choice of $\mathfrak{p}$.)

We define the notation $\Gamma$, $\Lambda$,
$X (F_\infty)$, $\mathfrak{X} (F_\infty)$, $M (F_\infty)$, $L (F_\infty)$, etc.~as
similar to Section \ref{sec_intro}.
In this appendix, we shall consider the following:

\begin{question} Let the notation be as in the previous paragraphs, and assume that $d$ satisfies (D2).
\begin{itemize}
\item[(Q1t)] When $X (F_\infty)$ is not trivial, is $X (F_\infty)_{\mathrm{fin}}$ not trivial?
\item[(Q2t)] When $X (F_\infty)$ is not trivial, is
$\Gal (M (F_\infty)/ L(F_\infty))$ not trivial?
\end{itemize}
\end{question}

Concerning the above questions, we shall show the following:

\begin{thm}\label{Th_p=2}
Assume that $d$ satisfies (D2).
\begin{itemize}
\item[(i)] If $d \equiv 5 \pmod{8}$, then (Q1t) has an affirmative answer for $E_*^d$.
\item[(ii)] Suppose that $d \equiv 1 \pmod{8}$.
If $\Gal (M (F_\infty)/ L(F_\infty))$ is not trivial, then $X (F_\infty)_{\mathrm{fin}}$ not trivial.
\end{itemize}
\end{thm}

Hence, in this situation, if (Q2t) has an affirmative answer then (Q1t) also has, 
and vice versa.
(Note that a similar assertion to that of stated in Remark \ref{rem_Q1toQ2} also holds.)

One can also show an analog of Theorem Y for this situation 
(see Section \ref{rem_ThB}).
Thus, (Q2t) has an affirmative answer when $\rank E_*^d (\mathbb{Q}) \geq 1$.
By combining this fact and Theorem \ref{Th_p=2}, we obtain the following:

\begin{cor}
Assume that $d$ satisfies (D2).
If $\rank E_*^d (\mathbb{Q}) \geq 1$, then (Q1t) has an affirmative answer for $E_*^d$.
\end{cor}

Furthermore, we can also show an analog of Theorem X.
That is, If $\rank E_*^d (\mathbb{Q}) \geq 2$, then $X (F_\infty)$ is infinite 
(see Section \ref{rem_ThB}).

\subsection{Proof of Theorem \ref{Th_p=2}}\label{sub_p2_proof}
Let the notation be as in Section \ref{sub_results}.
Recall that $F_0/K$ is totally ramified at $\mathfrak{p}$, and
$\mathfrak{P}$ is the unique prime of $F_0$ lying above $\mathfrak{p}$.
Let $\mathrm{cl} (\mathfrak{P})$ be the ideal class of $F_0$ containing $\mathfrak{P}$.
We note that the order of $\mathrm{cl} (\mathfrak{P})$ is equal to $1$ or $2$
because the class number of $K$ is $1$.

\begin{lem}\label{capitulation}
Assume that $d$ satisfies (D2).
If $\mathrm{cl} (\mathfrak{P})$ is not trivial (i.e., $\mathfrak{P}$ is not principal),
then $X (F_\infty)_{\mathrm{fin}}$ is not trivial.
\end{lem}

\begin{proof}
This assertion is essentially well known, and
this can be shown by using the arguments given in the papers treating
original Greenberg's conjecture.
(See also \cite[Corollary 3.5]{F-I} which treats a similar situation to ours.)

For $n \geq 0$, let $A (F_n)$ be the Sylow $2$-subgroup of the ideal class group of $F_n$,
and $D_n$ the subgroup of $A (F_n)$ consists of the classes
containing a power of the prime lying above $\mathfrak{P}$.
Assume that $\mathrm{cl} (\mathfrak{P})$ is not trivial.
As noted above, the order of $\mathrm{cl} (\mathfrak{P})$ is $2$.

In our situation, 
we can see that $|A (F_n)^{\Gal (F_n/F_0)}|$ is bounded with respect to $n$ 
(cf. the proof of \cite[Theorem 1]{Gre76}),
and then $|D_n|$ is also bounded.
From this, we can show that $\mathrm{cl} (\mathfrak{P})$ capitulates in $F_n$
if $n$ is sufficiently large (cf. \cite[p.267]{Gre76}).
Thus, by using \cite[p.218, Proposition]{Oza95}, 
we see that $X (F_\infty)_{\mathrm{fin}}$ is not trivial.
\end{proof}

We also show the following lemma.
(A similar result for the case of real quadratic fields is known.
See \cite[Lemma 2]{O-T97}.)

\begin{lem}\label{P_nontrivial}
Assume that $d$ satisfies (D2).
If $d$ has a rational prime divisor $\ell$ which satisfies
$\ell \equiv \pm 3 \pmod{8}$, then $\mathrm{cl} (\mathfrak{P})$ is not trivial.
\end{lem}

\begin{proof}
We put $\ell^* = \ell$ or $-\ell$ so that $\ell^*$ satisfies $\ell^* \equiv 1 \pmod{4}$.
Then $K (\sqrt{\ell^*})/K$ is unramified outside the primes lying above $\ell^*$, and
every prime of $K$ lying above $\ell^*$ is totally ramified in $K (\sqrt{\ell^*})$.

We note that every prime of $K$ lying above $\ell$ also ramifies in $F_0$.
Since the prime of $K$ lying above $7$ is ramified in $F_0$, 
$K (\sqrt{\ell^*})$ and $F_0$ are disjoint.
(See (\ref{eq_F0p2})).

Note that every prime of $K$ lying above $\ell$ is tamely ramified in $F_0 (\sqrt{\ell^*})$.
Then, we can see that $F_0 (\sqrt{\ell^*}) / F_0$ is an unramified extension
by combining the above results.

On the other hand, the rational prime $2$ is inert in $\mathbb{Q} (\sqrt{\ell^*})$.
Since $2$ splits in $K$ and $\mathfrak{p}$ ramifies in $F_0$,
we see that $\mathfrak{P}$ is inert in $F_0 (\sqrt{\ell^*})$.
This implies the assertion of the lemma.
\end{proof}

\begin{proof}[Proof of Theorem \ref{Th_p=2} (i)]
Since $d \equiv 5 \pmod{8}$, there is
a rational prime divisor $\ell$ of $d$ which satisfies
$\ell \equiv \pm 3 \pmod{8}$.
Then the theorem follows from Lemmas \ref{capitulation} and \ref{P_nontrivial}.
\end{proof}

We shall give a significant lemma to prove (ii).
Similar to the case when $p \geq 3$, we denote by $(F_0)_\mathfrak{P}$ the completion of $F_0$ at
$\mathfrak{P}$.
We also define $\mathcal{U}^1$, $E (F_0)^1$, and $\mathcal{E}^1$ similarly 
(see Section \ref{sub_criteria_pre}).

\begin{lem}\label{unit_p=2}
Assume that $d$ satisfies (D2) and $d \equiv 1 \pmod{8}$.
If $\mathfrak{P}$ is principal,
then $\mathcal{U}^1 / \mathcal{E}^1$ has no non-trivial $\mathbb{Z}_2$-torsion element.
\end{lem}

\begin{proof}
We mention that a quite similar result in a slightly different situation was given in
Li \cite{Li} (Theorem 1 (2) and Lemma 5 (2)).
That is, the filed $\mathbb{Q} (\sqrt[4]{-q})$ with a prime number $q$ 
satisfying $q \equiv 7 \pmod{16}$ was considered in \cite{Li}.
Our case is $F_0 = \mathbb{Q} (\sqrt[4]{-7 d^2})$ with $d \equiv 1 \pmod{8}$.
Our result can be also shown by using the same argument, and hence we only state an outline
of the proof.

By taking a suitable generator $\gamma$ of $\mathfrak{P}$, we can see that
the group of units of $F_0$ is generated by $-1$ and $\eta = \gamma^2 /2$
(see the proof of \cite[Lemma 5]{Li}).
Let $\mathrm{ord}_{\mathfrak{P}} (\cdot )$ be the normalized (additive) 
valuation at $\mathfrak{P}$.
By using a similar argument given in the proof of \cite[Lemma 5]{Li}, we see that
$\mathrm{ord}_{\mathfrak{P}} (\eta^2 -1) =2$.
Hence, we also see that
\[ \mathrm{ord}_{\mathfrak{P}} (\eta -1) =1 \qquad \text{and} \qquad
\mathrm{ord}_{\mathfrak{P}} (-\eta -1) =1. \]
We can see that the torsion units of $\mathcal{U}^1$ are $\pm 1$.
(Note that $(F_0)_{\mathfrak{P}}$ is isomorphic to $\mathbb{Q}_2 (\sqrt{3})$ when 
$d \equiv 1 \pmod{8}$. See also \cite{Li}.)
From these facts, 
we can see that $\mathcal{U}^1 / \mathcal{E}^1$ has no non-trivial 
$\mathbb{Z}_2$-torsion element.
\end{proof}

\begin{proof}[Proof of Theorem \ref{Th_p=2} (ii)]
If $\mathrm{cl} (\mathfrak{P})$ is not trivial, then $X (F_\infty)_\mathrm{fin}$ is not trivial
by Lemma \ref{capitulation}.
Hence, in the following, we assume that $\mathrm{cl} (\mathfrak{P})$ is trivial.

Similar to the proof of Proposition \ref{prop_U/E=1}, we use the argument given in 
\cite[Section 4]{Fuj}.
Under the above assumption, we see that
$\mathcal{U}^1 / \mathcal{E}^1$ has no non-trivial $\mathbb{Z}_2$-torsion element
by Lemma \ref{unit_p=2}.
From this, we see that $\mathfrak{X} (F_\infty)_\Gamma \cong
 X (F_\infty)_\Gamma$, and then 
\[ X (F_\infty)_{\mathrm{fin}}^\Gamma =X (F_\infty)^\Gamma \cong
\Gal (M (F_\infty)/L(F_\infty))_\Gamma. \]
(We used the validity of the $\{ \mathfrak{P} \}$-adic analog of Leopoldt's conjecture and 
the fact that $\mathfrak{X} (F_\infty)$ does not have a non-trivial
finite $\Lambda$-submodule. See \cite{Gre78}.)
The assertion follows from this.
\end{proof}

\subsection{Remarks}\label{rem_ThB}
Assume that $d$ satisfies (D2).
In our situation of this Appendix \ref{sec_appendix}, we can obtain an analog of Theorem Y.
That is, if $\rank E_*^d (\mathbb{Q}) \geq 1$ then $\Gal (M (F_\infty)/L(F_\infty))$ is 
not trivial.
This can be shown by using a similar method 
(see the proofs of Theorem 11 (p.231) and Lemmas 33, 35 of \cite{C-W}).
Note that the main difference from the case when $p \geq 3$ seems that certain cohomology 
groups (corresponding to $H^1 (G_\infty, \mathcal{E}_{\pi^{n+1}})$ in \cite{C-W}) 
are non-trivial (cf. also the proof of \cite[Lemma 2.7]{Gon}). 
However, by using Sah's lemma, 
we can see that the orders of these groups are 
bounded with respect to $n$, and hence this does not give an essential difficulty.
(See, e.g., \cite{Sah}, \cite[Lemma A.2]{B-R}. 
See also \cite[(2.2)Lemma]{Rub87}.)

We can also obtain an analog of Theorem X by using the known method 
for the case when $p \geq 3$.
(See \cite[p.551, Remark]{Rub87}, \cite[pp.364--366]{Gre01}.
See also \cite{FKY} for a more detailed argument.)
We put $r = \rank E_*^d (\mathbb{Q})$, and assume that $r \geq 2$.
Then, for every sufficiently large $n$, 
we can construct an unramified extension $L_n/F_n$ whose Galois group 
is isomorphic to $(\mathbb{Z} / 2^{n-c} \mathbb{Z})^{\oplus r-1}$, 
where $c$ is a constant which does not depend on $n$.
We can show this assertion by imitating the argument given in \cite{FKY}.
However, as similar to the above paragraph, it is necessary to pay attention to  
the difference which comes from the situation that $p=2$.
In particular, $H^1 (\Gal (F_n/K), E_*^d [\pi^{n+2}])$ is not trivial 
(see., e.g., the proof of \cite[Lemma 2.7]{Gon}).

As noted in \cite{Gon}, \cite{CLT} (see also \cite[Theorem 19.1.1]{Gross}), 
it is known that $L (E_*^d, 1) =0$ when $d <0$ 
(recall that $d \not\equiv 0 \pmod{7}$).
Hence, $\rank E_*^d (\mathbb{Q})$ is expected to be positive for this case.
We also mention that a sufficient condition to satisfy $\rank E_*^d (\mathbb{Q})=1$ is
given in \cite[Theorem 1.4]{CLT}.

We shall give another remark.
The $\mathbb{Z}_2$-rank of $\mathfrak{X} (F_\infty)$ was considered in \cite{C-C}.
See also the recently announced preprint \cite{Li2}.
These results seem helpful for future research on our questions (Q1t), (Q2t).

\begin{acknowledgments}
This research was born out of the discussions with Satoshi Fujii.
In particular, 
informing the ideas concerning the results of \cite{Fuj} from him 
was one of the motivations to start this study.
(See also Remark \ref{rem_Maire}, the proofs of Proposition \ref{prop_U/E=1} and Theorem \ref{Th_p=2} (ii).)
The author would express thanks to him. 
The author also would express his gratitude to Takashi Fukuda for his kind responses to 
the inquiries about computation, and to Keiichi Komatsu for his encouragement.
They also gave significant comments on an earlier manuscript. 
The author also would like to express thanks to Christian Maire 
for giving comments on an earlier manuscript (Remark \ref{rem_Maire}) 
and to Jianing Li for sending a manuscript of a stronger version of \cite{Li} and 
giving information about his results. 
The author also would like to express his gratitude to the referee of 
an earlier manuscript (who seems different from the person mentioned in Remark \ref{rem_old}). 
His/her comments motivated the author to improve the contents of this paper.
This work was partly supported by JSPS KAKENHI Grant Number JP15K04791.
\end{acknowledgments}

\bigskip

\begin{flushleft}
Tsuyoshi Itoh \\
Division of Mathematics, 
Education Center,
Faculty of Social Systems Science, \\
Chiba Institute of Technology, \\
2--1--1 Shibazono, Narashino, Chiba, 275--0023, Japan \\
e-mail : \texttt{tsuyoshi.itoh@it-chiba.ac.jp}

\end{flushleft}

\end{document}